\documentclass[11pt]{article}

\usepackage{mathrsfs}

\usepackage[T1]{fontenc}
\usepackage[utf8]{inputenc}
\usepackage{lmodern}
\usepackage{amssymb,amsmath,amsthm}
\usepackage{bbm}
\usepackage{graphicx}
\usepackage{float}
\usepackage{xcolor}
\usepackage[colorlinks,linkcolor=black!70!black,citecolor=black]{hyperref} % Or: hidelinks
\usepackage[a4paper,margin=1.5cm]{geometry}
\usepackage{tikz}
\usepackage{tikz-cd}

\usetikzlibrary{decorations.pathreplacing, patterns,shapes,snakes}

\usepackage{pgfplots}

\long\def\metanote#1#2{{\color{#1}\
\ifmmode\hbox\fi{\sffamily\mdseries\upshape [#2]}\ }}

\newcommand{\ra}{\rightarrow}

\newcommand{\thh}{\text{th}}

\newcommand{\be}{\begin{equation}}
\newcommand{\ee}{\end{equation}}
\newcommand{\bi}{\begin{itemize}}
\newcommand{\ei}{\end{itemize}}

\newcommand{\m}{\mbox{$\bf m $}}

\newcommand{\Law}{{\mathcal{L}}}

\newcommand{\calT}{{\mathcal{T}}}
\newcommand{\calM}{{\mathcal{M}}}

\newcommand{\calO}{{\mathcal{O}}}

\newcommand{\calP}{{\mathcal{P}}}

\newcommand{\Nm}{{\mathbb{N}}}

\newcommand{\Rm}{{\mathbb R}}
\newcommand{\Pm}{{\mathbb P}}

\newcommand{\expE}{{\mathbb E}}

\newcommand{\Ind}{{\mathbbm{1}}}

\newtheorem{theo}{Theorem}[section]

\newtheorem{defin}[theo]{Definition}

\newtheorem{prop}[theo]{Proposition}
\newtheorem{cor}[theo]{Corollary}

\newtheorem{assum}[theo]{Assumption}

\newtheorem*{rep@theorem}{\rep@title}
\newcommand{\newreptheorem}[2]{%
\newenvironment{rep#1}[1]{%
 \def\rep@title{#2 \ref{##1}}%
 \begin{rep@theorem}}%
 {\end{rep@theorem}}}
\makeatother
\newreptheorem{theorem}{Theorem}
\newreptheorem{lemma}{Lemma}
\newreptheorem{defin}{Definition}
\newreptheorem{cond}{Condition}
\newreptheorem{prop}{Proposition}
\newreptheorem{lem}{Lemma}
\newreptheorem{cor}{Corollary}
\newreptheorem{rmk}{Remark}
\newreptheorem{exer}{Exercise}
\newreptheorem{conj}{Conjecture}
\newreptheorem{assum}{Assumption}
\newreptheorem{equation}{Equation}

\long\def\metanote#1#2{{\color{#1}\
\ifmmode\hbox\fi{\sffamily\mdseries\upshape [#2]}\ }}

\begin{document}
\setcounter{page}{1}

\title{Stochastic Approximation of the Paths of Killed Markov Processes Conditioned on Survival}
\author{Oliver Tough}
\date{\today}

\maketitle

\begin{abstract}
Reinforced processes are known to provide a stochastic representation for the quasi-stationary distribution of a given killed Markov process - describing the killed Markov process at fixed time instants. In this paper we shall adapt the construction to provide a pathwise description. We also obtain a stochastic approximation for the quasi-limiting distributions of reducible killed Markov processes as a corollary.
\end{abstract}

\section{Introduction}\label{section:introduction}

We consider in discrete time a killed Markov process $(X_t)_{0\leq t<\tau_{\partial}}$ on the state space $\chi\cup \partial$, evolving in $\chi$ until the killing time $\tau_{\partial}=\inf\{t>0:X_{t}\in \partial\}$, after which time it remains in the cemetery state (we assume without loss of generality that $\partial$ is a one-point set). Note that whilst we specify that time is discrete, the results of this paper may be applied in continuous time by discretising time. 

In general, one is interested in the law of the killed Markov process (with initial condition $X_0\sim\mu\in\calP(\chi)$) conditioned on survival,
\[
\Law_{\mu}(X_t\lvert \tau_{\partial}>t),
\]
and the long-time limits of this law,
\[
\Law_{\mu}(X_t\lvert \tau_{\partial}>t)(\cdot)\ra \pi(\cdot)\quad \text{as}\quad t\ra\infty.
\]
General conditions for the existence and uniqueness of these limits are given by \cite{Champagnat2014} and \cite{Champagnat2018}. In general, these limits correspond to quasi-stationary distributions (QSDs) $\pi$,
\[
\Law_{\pi}(X_t\lvert \tau_{\partial}>t)(\cdot)=\pi(\cdot)\quad\text{for all}\quad 0\leq t<\infty.
\] 
Aldous, Flannery and Palacios \cite{Aldous1988} introduced a method for simulating QSDs based on reinforced processes. The reinforced process, $(Y_t)_{0\leq t<\infty}$, is obtained by running a copy of $X_t$ until it is killed,
\begin{equation}\label{eq:reinforced process up to the killing time}
(Y_t)_{0\leq t<\tau_{\partial}}=(X_t)_{0\leq t<\tau_{\partial}}.
\end{equation}
At this killing time, $Y_t$ jumps to a point sampled independently from the empirical measure of the history,
\begin{equation}\label{eq:reinforced process sampling from the history}
Y_{\tau_{\partial}}\sim \frac{1}{\tau_{\partial}}\sum_{s=0}^{\tau_{\partial}-1}\delta_{Y_s}(\cdot).
\end{equation}

This is then repeated inductively, so that at the $n^{\thh}$ killing time $\tau^n_{\partial}$, $Y_t$ jumps to a point sampled independently from the empirical measure of the history of $Y_t$ up to that killing time,
\begin{equation}\label{eq:nth sampling reinforced process}
Y_{\tau^n_{\partial}}\sim \frac{1}{\tau^n_{\partial}}\sum_{s=0}
^{\tau^n_{\partial}-1}\delta_{Y_s}(\cdot)ds,
\end{equation}
and continues evolving like a copy of $X_t$ as before. In \cite{Aldous1988} they established that these provide an approximation method for the QSDs of irreducible killed Markov chains when the state space is finite and time is discrete, proving that
\begin{equation}\label{eq:limit of historical measure of reinforced processes}
\frac{1}{t}\sum_{s=0}^{t-1}\delta_{Y_s}(\cdot)ds\overset{a.s.}{\ra} \pi(\cdot)\quad \text{as}\quad t\ra\infty.
\end{equation}
This was made quantitative for a more general version of this algorithm by Benaim and Cloez \cite{Benaim2015}, and has since been extended to a quite general setting in \cite{Mailler2020} and \cite{Benaim2019}. These results do not, however, apply to reducible killed Markov processes. We shall obtain in Corollary \ref{cor:convergence of reinforced with renewal} a stochastic approximation for the quasi-limiting distributions of reducible killed Markov processes for given initial condition as a corollary of our main result, Theorem \ref{theo:main theorem}.

Note, however, that QSDs only describe killed Markov processes at fixed time instants. One may also be interested in obtaining pathwise information, so we may seek to approximate
\[
\Law_{\mu}((X_s)_{0\leq s\leq t}\lvert \tau_{\partial}>t)(\cdot)
\]
for finite $t$. In this note, we shall demonstrate how the construction of reinforced processes may be adapted to provide such a pathwise approximation. Since our result (Theorem \ref{theo:main theorem}) is restricted to discrete time, for continuous time processes we obtain (for any $m<\infty$) an approximation for
\[
\Law((X_0,X_{\frac{t}{m}},X_{\frac{2t}{m}},\ldots,X_t)\lvert \tau_{\partial}>t).
\]

One may also consider the $Q$-process, which provides a pathwise description of $(X_t)_{0\leq t<\tau_{\partial}}$ conditioned never to be killed, a definition of which is given in \cite[Theorem 3.1]{Champagnat2014}. 

A second method for approximating QSDs is given by the Fleming-Viot process, a particle system introduced by Burdzy, Holyst and March in \cite{Burdzy2000}. They considered the case whereby the killed Markov process is Brownian motion in an open, bounded domain, killed instantaneously upon contact with the boundary. They established in \cite{Burdzy2000} that this particle system provides an approximation method for both the distribution of this killed Brownian motion conditioned on survival at fixed instants of time, and the corresponding QSD. This was later extended to a general setting by Villemonais \cite{Villemonais2011}. In \cite{Bieniek2018}, Bieniek and Burdzy established that the Fleming-Viot process also provides for the distribution of the path of a killed Markov process conditioned to survive over a fixed time interval. Since then, the present author established in \cite[Corollary 5.1]{Tough2021} that the Fleming-Viot process also provides a representation for the $Q$-process when the killed Markov process is a normally reflected diffusion in a compact domain, killed at position-dependent Poisson rate (this result may be extended to a more general setting, subject to overcoming an additional difficulty if the state space is non-compact). 

Thus, whilst both reinforced processes and the Fleming-Viot process are known to provide an approximation method for the QSDs of killed Markov processes - thereby describing killed Markov processes at fixed time instants - prior to the present paper only the Fleming-Viot process was known to provide a pathwise description. In the present paper, we shall show how the construction of reinforced processes may be adapted to provide such a pathwise description.

Whereas reinforced processes are constructed by sampling a spatial location from the history, we adapt this construction by sampling both the spatial and temporal location (or, with some probability, sampling a killed Markov process started from some fixed initial distribution at time $0$ instead). We refer to the resultant constructions as reinforced path processes. A more precise definition is given by the following.

Throughout this paper we abuse notation by writing $[a,b]$ for $[a,b]\cap \Nm$, for all $a,b\in \Nm$. For paths $f$ and $g$ on the time intervals $I$ and $J$ respectively such that $f=g$ on $I\cap J$, we define
\begin{equation}\label{oplus}
f\oplus g:I\cup J\ni t\mapsto \begin{cases}
f(t),\quad t\in I\\
g(t),\quad t\in J
\end{cases}.
\end{equation}

\begin{defin}[Reinforced Path Process]\label{defin:reinforced path process}
We fix a time horizon $0<T\leq \infty$, a renewal probability $0<p<1$ and initial condition $\mu\in\calP(\chi)$. The reinforced path process $(u_n)_{n=1}^{\infty}$ is defined by inductively sampling triples $u_n=(t_b^n,f^n,t^n_d)$, whereby $t_b^n$ is the $n^{\text{th}}$ birth time, $t^n_d$ is the $n^{\text{th}}$ killing time, and $f^n$ is the $n^{\text{th}}$ path from time $0$ to time $t^n_d$. Whilst $u_n$ is considered to be "alive" only between times $t^n_b$ and $t^n_d$, its path $f^n$ is defined prior to time $t^n_b$ - it includes an "ancestral path". The first triple $u_1=(t^1_b,f^1,t^1_d)$ is defined by taking a copy $(X_t)_{0\leq t<\tau_{\partial}}$ of the killed Markov process with initial condition $X_0\sim \mu$, and defining $u_1=(t^1_b,f^1,t^1_d)$ to be $(0,(X_{t\wedge (\tau_{\partial}-1)})_{0\leq t\leq T},\tau_{\partial}\wedge (T+1))$. 

Given $u_1,\ldots,u_n$ we inductively define $u_{n+1}=(t^{n+1}_b,f^{n+1},t^{n+1}_d)$ as follows. With probability $p$, we take another independent copy $(X_t)_{0\leq t<\tau_{\partial}}$ of the killed Markov process with initial condition $X_0\sim \mu$, and define $u_{n+1}=(t^{n+1}_b,f^{n+1},t^{n+1}_d)$ to be $(0,(X_{t\wedge (\tau_{\partial}-1)})_{0\leq t\leq T},\tau_{\partial}\wedge (T+1))$, in which case we say that we "renew". Otherwise, with probability $1-p$, we choose $m\in [1,n]$ independently with probability 
\[
\frac{t^m_d-t^m_b}{\sum_{1\leq \ell\leq n}(t^{\ell}_d-t^{\ell}_b)}.
\]
Given our choice of $m$, we then choose $t'\in [t^m_b,t^m_d-1]$ independently at random. Given the choice of $m$ and $t'$, we independently take $(X_t)_{0\leq t<\tau_{\partial}}$ to be a copy of the killed Markov process with initial condition $X_0=f^m(t')$, and set
\[
u_{n+1}=(t^{n+1}_b,f^{n+1},t^{n+1}_d):=(t',f^m_{\lvert_{[0,t']}}\oplus (X_{(t-t')\wedge (\tau_{\partial}-1)})_{t'\leq t\leq T},(t'+\tau_{\partial})\wedge (T+1)).
\]

\end{defin}

Clearly, one may formulate the same construction in continuous time. We depict the reinforced path process corresponding to Brownian motion killed at the boundary of a bounded interval in Figure \ref{fig:renewal tree simple version}.
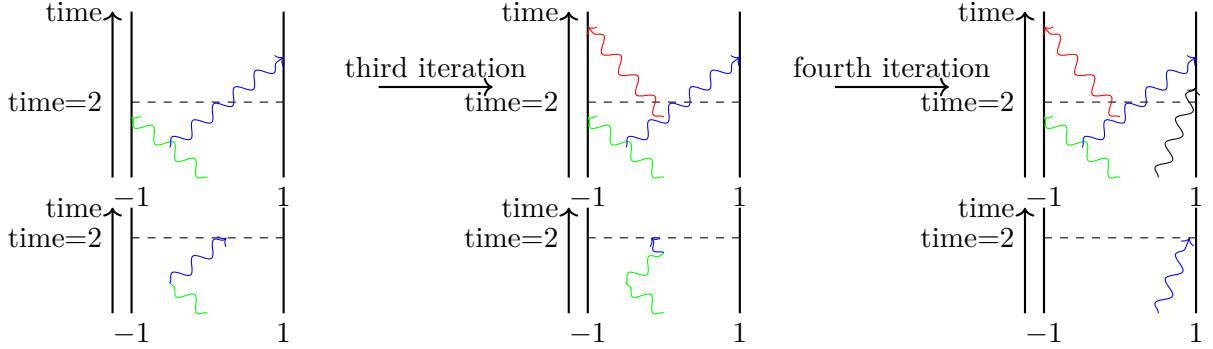
\begin{figure}[h]
  \begin{center}
    \begin{tikzpicture}[xscale=0.5,yscale=0.4]
\foreach \y [count=\n]in {-5,7,19}{
\draw[-,dashed,black](-7+\y,2.5) -- ++ (4,0);
\node[left] at (-7.5+\y,2.5) {time=$2$};
\draw[-,thick,black](-7+\y,0) -- ++(0,5.5); 
\draw[-,thick,black](-3+\y,0) -- ++(0,5.5);
\draw[->,thick,black](-7.5+\y,0) -- ++(0,5.5);
\node[left] at (-7.5+\y,5.5) {time};

 			\draw[->,snake=snake,green](-5+\y,0) -- ++(-2,2);

			\draw[->,snake=snake,blue](-6+\y,1) -- ++(3,3);
\node[below] at (-7+\y,0) {$-1$};
\node[below] at (-3+\y,0) {$1$};}
\foreach \y [count=\n]in {-5}{
\node[above] at (\y+1,3) {third iteration};
\node[above] at (\y+13,3) {fourth iteration};
\draw[->,thick,black](-0.5+\y,3) -- ++(3,0); 
\draw[->,thick,black](-0.5+12+\y,3) -- ++(3,0);
}
\foreach \y [count=\n]in {7}{
\draw[->,snake=snake,black](\y+8,0) -- ++(1,3);
\foreach \z [count=\n]in {0,12}{
			\draw[->,snake=snake,red](-5+\y+\z,2) -- ++(-2,3);}			}
\foreach \y [count=\m] in {-4.5}{
\foreach \x [count=\n]in {-5,7,19}{
\draw[-,dashed,black](-7+\x,2.5+\y) -- ++ (4,0);
\node[left] at (-7.5+\x,2.5+\y) {time=$2$};
\draw[-,thick,black](-7+\x,\y) -- ++(0,3.5); 
\draw[-,thick,black](-3+\x,\y) -- ++(0,3.5);
\draw[->,thick,black](-7.5+\x,\y) -- ++(0,3.5);
\node[left] at (-7.5+\x,3.5+\y) {time};

\node[below] at (-7+\x,\y) {$-1$};
\node[below] at (-3+\x,\y) {$1$};
}}

\foreach \y [count=\m] in {-4.5}{
\foreach \x [count=\n]in {7}{
 			\draw[-,snake=snake,green](-5+\x,\y) -- ++(-1,1);
			\draw[-,snake=snake,green](-6+\x,1+\y) -- ++(1,1);
			\draw[->,snake=snake,blue](-5+\x,2+\y) -- ++(-2/6,0.5);			
}}
\foreach \y [count=\m] in {-4.5}{
\foreach \x [count=\n]in {-5}{
 			\draw[-,snake=snake,green](-5+\x,\y) -- ++(-1,1);
			\draw[->,snake=snake,blue](-6+\x,1+\y) -- ++(1.5,1.5);
}}

\foreach \y [count=\m] in {-4.5}{
\foreach \x [count=\n]in {7}{
\draw[->,snake=snake,blue](\x+8,\y) -- ++(5/6,2.5);
}}
		    \end{tikzpicture}
    \caption{The killed Markov process here is Brownian motion in $(-1,1)$, killed instantaneously at $\{-1,1\}$. The third and fourth iteration are shown above, with the first, second, third and fourth paths in green, blue, red and black respectively. We may see that in the second and third iteration we sample from the history, whilst with the fourth we begin a new particle from time $0$ - we renew. While each path is considered to be alive only between the time it is born (the sampled time) and the time it is killed at the boundary, it also carries with it information of its ancestral path from time $0$ to the time it is born. Thus the three paths corresponding to time $2$ after $4$ iterations are shown below (the first path is not alive at time $2$), with the path while they are alive in blue and the ancestral path before they are born in green.}
  \label{fig:renewal tree simple version}
  \end{center}
\end{figure}

\subsection*{Structure of the paper}

We will provide a statement of Theorem \ref{theo:main theorem} in Section \ref{section:theorem statement}, which establishes that reinforced path processes provide an approximation for $\Law_{\mu}((X_s)_{0\leq s\leq t}\lvert \tau_{\partial}>t)(\cdot)$, thus giving a pathwise description of killed Markov processes conditioned on survival. We will then use this in Corollary \ref{cor:convergence of reinforced with renewal} to provide an approximation for the quasi-limiting distributions of reducible killed Markov processes.

We provide a proof of Theorem \ref{theo:main theorem} in Section \ref{section:theorem proof inf time}.

\section{Statement of Results}\label{section:theorem statement}

For any topological space $\calT$, we write $\calM(\calT)$ and $\calP(\calT)$ for the set of Borel measures (respectively Borel probability measures) on $\calT$, equipped with the topology of weak convergence of measures. 

We assume that $(\chi,d)$ is a metric space. We define $K:\chi\ra \calM(\chi)$ to be a submarkovian transition kernel, which defines the discrete-time killed Markov process $(X_t)_{0\leq t<\tau_{\partial}}$. We fix an initial condition $\mu\in\calP(\chi)$, (possibly infinite) time horizon $T\in \Nm\cup\{\infty\}$ and renewal probability $0<p\leq 1$. We take a reinforced path process $(u_n)_{n=1}^{\infty}=((t^n_b,f^n,t^n_d))_{n=1}^{\infty}$ - that is a solution to Definition \ref{defin:reinforced path process} - corresponding to these choices.

We write $\bar \chi$ for the completion of $\chi$. We define $C_b(\chi)$, $C_0(\bar \chi)$ and $C_0(\chi)$ to be the space of bounded, continuous functions on $\chi$ (respectively continuous functions on $\bar \chi$ which vanish on $\partial \chi:=\bar\chi\setminus \chi$, and the restriction to $\chi$ of elements of $C_0(\bar \chi)$), all equipped with the uniform norm.

We impose the following assumption.

\begin{assum}\label{assum:ext to compact completion}
The completion $\bar \chi$ is compact. Moreover we assume that $(\chi\ni x\mapsto Kf(x):=K(x,f))\in C_0(\chi)$ for all $f\in C_b(\chi)$. Furthermore if $T<\infty$ we assume that $\Pm_{\mu}(\tau_{\partial}>T)>0$, whilst if $T=\infty$ we assume that  $\Pm_{\mu}(\tau_{\partial}>t)>0$ for all $t<\infty$.
\end{assum}

Thus $K$ defines a contraction operator 
\[
K:\chi\ni x\mapsto Kf(x):=K(x,f)\in C_0(\chi)\subseteq C_b(\chi)
\]
with spectral radius $r(K)$. If $T=\infty$ then we impose the following additional assumption.
\begin{assum}\label{assum:inf time horizon}
The operator $K:C_b(\chi)\ra C_b(\chi)$ is compact, and $\Pm_x(\tau_{\partial}=\infty)=0$ for all $x\in\chi$.
\end{assum}

\begin{theo}\label{theo:main theorem}
There exists a unique solution, $Z$, to
\begin{equation}\label{eq:fixed point eqn for Z}
z=p\sum_{s=0}^T\Pm_{\mu}(\tau_{\partial}>s)\Big[1-\frac{1-p}{z}\Big]^{-(s+1)},\quad z\in [1,\infty).
\end{equation}

We define the coefficients $(\gamma_t)_{0\leq t\leq T}$ (or $(\gamma_t)_{0\leq t<\infty}$ if $T=\infty$) to be
\begin{equation}\label{eq:formula for gamma}
\gamma_t=p\Pm_{\mu}(\tau_{\partial}>t)\Big[1-\frac{1-p}{Z}\Big]^{-(t+1)},
\end{equation}
whereby $Z$ is the unique solution to \eqref{eq:fixed point eqn for Z}. Then for any $t\in [0,T]$ (or $t\in \Nm$ if $T=\infty$) we have
\begin{equation}\label{eq:convergence in theorem statement}
\frac{1}{N}\sum_{n=1}^N\Ind(t^n_b\leq t<t^n_d)\delta_{f^n_{[0,t]}}(\cdot)\overset{\calM(F([0,t];\chi))}{\ra} \gamma_t\Law_{\mu}((X_s)_{0\leq s\leq t}\lvert \tau_{\partial}>t)(\cdot)\quad \text{almost surely as}\quad N\ra\infty.
\end{equation}
\end{theo}

Note that if we have a continuous-time killed Markov process $(X_t)_{0\leq t<\tau_{\partial}}$, the above theorem can be applied to obtain, for any $m\in \Nm_{>0}$ and $t\in \Rm_{>0}$, a stochastic approximation for 
\[
\Law_{\mu}((X_0,X_{\frac{t}{m}},X_{\frac{2t}{m}},\ldots,X_t)\lvert \tau_{\partial}>t).
\]

\subsection*{Choosing $p$}

We may ask how $p$ should be chosen. We consider the case whereby $T<\infty$, and assume that the killed Markov process satisfies the consequence of \cite[Theorem 2.1]{Champagnat2017}. Thus we assume that there exists a right eigenfunction $\phi$, positive and bounded on $\chi$, such that 
\begin{equation}\label{eq:asymptotics of killing probability}
r(K)^{-t}\Pm_{\mu}(\tau_{\partial}>t)\ra \phi(\mu)\quad\text{exponentially quickly, uniformly over all}\quad \mu\in\calP(\chi).
\end{equation}

Note that if $\gamma_T$ is small, then by considering the mass on both sides of \eqref{eq:convergence in theorem statement} we see that for large $n$ a small proportion of the $u_ns$ will be alive at time $T$, leading to a large variance. If, on the other hand, $\gamma_t$ is small for some $t<T$, then all $f^ns$ at time $T$ will share a small number of ancestral paths at time $t$, leading to a large variance. Thus it is reasonable that we should seek to maximise $\min_{0\leq t\leq T}\gamma_t$. Therefore by \eqref{eq:asymptotics of killing probability} it is reasonable to choose $p$ such that
\begin{equation}
1-\frac{1-p}{Z}=r(K).
\end{equation}
Straightforward algebra shows that this is achieved by
\begin{equation}
p=[1+(1-r(K))\sum_{s=0}^T\Pm_{\mu}(\tau_{\partial}>s)r(K)^{-(s+1)}]^{-1}=\frac{r(K)}{(1-r(K))\mu(\phi)T}+\calO_{T\ra\infty}(\frac{1}{T^2}).
\end{equation}

\subsection*{Approximation of the QLDs of reducible killed Markov processes}

If a killed Markov process is reducible, even if the state space is finite it is an open problem to determine which QSDs (if any at all) will be obtained from the reinforced processes $(Y_t)_{0\leq t<\infty}$ described in the introduction, by taking the limit \eqref{eq:limit of historical measure of reinforced processes}. We suppose that for some $\mu,\pi\in\calP(\chi)$ we have
\[
\Law_{\mu}(X_t\lvert \tau_{\partial}>t)\ra \pi(\cdot)\quad\text{as}\quad  t\ra\infty.
\]
We also impose assumptions \ref{assum:ext to compact completion} and \ref{assum:inf time horizon}. The following corollary, which came from a discussion of the author with Michel Benaim, allows us to obtain the quasi-limiting distribution $\pi(\cdot)$. We construct for $0<p\leq 1$ a reinforced process with renewal $(Y^p_t)_{0\leq t<\infty}$ as follows. We firstly take a copy of the killed Markov process $(X_t)_{0\leq t<\tau_{\partial}}$ with initial condition $X_0\sim \mu$, and define $(Y^p_t)_{0\leq t<\tau_{\partial}}$ as in \eqref{eq:reinforced process up to the killing time}. At this killing time, with probability $p$ we "renew", sampling
\begin{equation}
Y^p_{\tau_{\partial}}\sim \mu(\cdot).
\end{equation}
Otherwise, with probability $1-p$, we sample from the empirical measure of the history,
\begin{equation}
Y^p_{\tau_{\partial}}\sim \frac{1}{\tau_{\partial}}\sum_{t=0}^{\tau_{\partial}-1}\delta_{Y^p_{t}}(\cdot),
\end{equation}
as in \eqref{eq:reinforced process sampling from the history}. The process $Y^p_t$ then continues evolving like a copy of $(X_t)_{0\leq t<\tau_{\partial}}$ up to its next killing time. This is then repeated inductively, so that at each killing time we sample from $\mu$ with probability $p$, otherwise sampling from the empirical measure of the history as in \eqref{eq:nth sampling reinforced process}. We are then able to obtain the quasi-limiting distribution $\pi(\cdot)$ from $(Y^p_t)_{0\leq t<\infty}$, for small $0<p\ll 1$.

\begin{cor}\label{cor:convergence of reinforced with renewal}
For given $0<p\leq 1$ we let $Z^p$ be the unique solution to \eqref{eq:fixed point eqn for Z} for $T=\infty$ and $(\gamma^p_t)_{0\leq t<\infty}$ be the coefficients thereby defined in \eqref{eq:formula for gamma}. Then we have
\begin{equation}
\frac{1}{t}\sum_{s=0}^{t-1}\delta_{Y^p_s}(\cdot)\overset{\text{a.s.}}{\ra}\pi^p(\cdot)\quad\text{as}\quad t\ra\infty,
\end{equation}
whereby
\begin{equation}
\pi_p(\cdot):=\sum_{t=0}^{\infty}\gamma^p_t\Law_{\mu}(X_t\lvert\tau_{\partial}>t)(\cdot)\ra \pi(\cdot)\quad\text{as}\quad p\ra 0.
\end{equation}
\end{cor}

\section{Proof of Theorem \ref{theo:main theorem}}\label{section:theorem proof inf time}
We recall that $\bar \chi$ is the completion of $\chi$, which by assumption is compact. We extend $K$ to $\bar \chi$ by setting $K(x,\cdot)=0$ for $x\in \partial\chi=\bar \chi\setminus \chi$, labelling this extended submarkovian kernel as $K$ by abuse of notation.

For (discrete and finite) time intervals $[t_1,t_2]\subseteq \Nm$, we write $F([t_1,t_2];\bar \chi)$ for the set of functions $[t_1,t_2]\ra \bar \chi$, which we equip with the uniform metric 
\[
d_{F;[t_1,t_2]}(f,g)=\sup_{t\in [t_1,t_2]}d(f(t),g(t)).
\]

We defer for later the proof of the following proposition.

\begin{prop}\label{prop:exist of a fixed point Z inf time horizon}
There exists a unique solution, $Z$, to \eqref{eq:fixed point eqn for Z}.
\end{prop}
We will establish convergence by formulating the reinforced path process as an urn process, and applying \cite[Theorem 1]{Mailler2020}. We use the terminology given in \cite[Section 1.1]{Mailler2020} throughout. 

To identify the limit as being of the desired form, the following observation shall be crucial. We fix $0\leq t_1\leq t_2$. Take $(Y_u)_{0\leq u\leq t_1}$ and $(Z_u)_{t_1\leq u\leq t_2}$ such that $(Y_u)_{0\leq u\leq t_1}\sim \Law_{\mu}((X_u)_{0\leq u\leq t_1}\lvert \tau_{\partial}>t_1)$ and, conditional on $(Y_u)_{0\leq u\leq t_1}$, $(Z_u)_{t_1\leq u\leq t_2}\sim \Law_{Y_{t_1}}((X_{u-t_1})_{t_1\leq u\leq t_2})$. Then we have
\begin{equation}\label{eq:identity for law of path started from other path}
Y\oplus Z\sim\Law_{\mu}((X_u)_{0\leq u\leq t_2}\lvert \tau_{\partial}>t_1).
\end{equation}

\subsection*{The urn process formulation}

We must distinguish between the $T<\infty$ and $T=\infty$ cases. If $T=\infty$ we fix arbitrary $\bar T\in\Nm$ and define $\ast$ to be a point distinguished from $\Nm$. We define
\[
\begin{split}
E_t:=\{t\}\times F([0,t];\bar \chi)\quad\text{for all}\quad t\in \Nm,\quad
E_{\ast}:=\{\ast\}\times \bar \chi.
\end{split}
\]
We then define
\[
\begin{split}
E:=\cup_{t\in [0,T]}E_t\quad\text{if}\quad T<\infty,\quad E:=\cup_{t\in [0,T]}E_t\cup E_{\ast}\quad\text{if}\quad T=\infty,
\end{split}
\]
which we equip with the metric $d_E$ defined by
\[
d_E((t,f),(s,g))=\begin{cases}
1,\quad t\neq s\\
1\wedge d_{F;[0,t]}(f,g),\quad t=s\leq T\\
1\wedge d(f,g),\quad t=s=\ast
\end{cases},
\]
under which $E$ is a compact metric space. We recall that $(u_n)_{n=1}^{\infty}=((t^n_b,f^n,t^n_d))_{n=1}^{\infty}$ is our reinforced path process. We define
\[
\begin{split}
m_N(\cdot):=\sum_{n=1}^N v_n(\cdot)\in\calM(E)\quad\text{whereby}\\ v_n(\cdot):=\begin{cases}\sum_{t=t^n_b}^{t^n_d-1}\Ind(t\leq T)\delta_{(t,f^n_{\lvert_{[0,t]}})}(\cdot),\quad T<\infty\\
\sum_{t=t^n_b}^{t^n_d-1}\big(\Ind(t\leq \bar T)\delta_{(t,f^n_{\lvert_{[0,t]}})}(\cdot)+\Ind(t>\bar T)\delta_{(\ast,f^n(t))}(\cdot)\big),\quad T=\infty
\end{cases}.
\end{split}
\]

If $T<\infty$, we define $\Gamma=T+1$. If $T=\infty$, on the other hand, then $(\{x\in \chi:K^t1(x)=1\})_{t=1}^{\infty}$ is a descending sequence of compact sets, whose intersection must be empty by Assumption \ref{assum:inf time horizon}. Therefore $\lvert\lvert K^t\rvert\rvert_{\text{op}}=\sup_{x\in\chi}K^t1(x)<1$ for some $t<\infty$ large enough, so that the spectral radius of $K$, $r(K)$, is less than $1$. Thus
\begin{equation}\label{eq:spectral radius of K less than 1}
\limsup_{t\ra\infty}(\sup_{x\in\chi}\Pm_x(\tau_{\partial}>t))^{\frac{1}{t}}=\limsup_{t\ra\infty}(\lvert\lvert K^t\rvert\rvert_{\text{op}})^{\frac{1}{t}}=r(K)<1,
\end{equation}
so that $\sup_{x\in\chi}\expE_x[\tau_{\partial}]<\infty$. Therefore we may define $\Gamma:=1+\sup_{x\in \chi}\expE_x[\tau_{\partial}]$ in the case that $T=\infty$.

We observe that $(m_N)_{N\geq 1}$ is a measure-valued Polya process with:
\begin{itemize}
\item
Initial composition
\[
m_1=\begin{cases}\sum_{t=0}^{\tau_{\partial}-1}\delta_{(t,(X^{\mu}_{s})_{0\leq s\leq t})}\\
\sum_{t=0}^{\tau_{\partial}-1}\big(\delta_{(t,(X^{\mu}_{s})_{0\leq s\leq t})}\Ind(t\leq \bar T)+\delta_{(\ast,X^{\mu}_t)}\Ind(t>\bar T)\big)
\end{cases},
\]
whereby $(X^{\mu}_t)_{0\leq t<\tau_{\partial}}$ is an independent copy of the killed Markov process with submarkovian transition kernel $K$ and initial condition $X^{\mu}_0\sim \mu$.
\item
Independent and identically distributed random replacement kernels 
\[
E\ni (t,f)\mapsto R^n((t,f);.)\in \calM(E),\quad 1\leq n<\infty,
\]
defined as follows. We take for each $n$, independently of each other and everything else, a Bernoulli random variable $B\sim \text{Ber}(p)$, a copy $(X^{\mu}_t)_{0\leq t<\tau_{\partial}}$ of the killed Markov process with submarkovian transition kernel $K$ and initial condition $X_0\sim \mu$, and a family of copies $\{(X^x_t)_{0\leq t<\tau_{\partial}}:x\in \chi\}$ of the same killed Markov process with initial conditions $X^x_0=x$. Note that this last definition makes sense since there exists a probability space $(\Omega,\Pm)$ and a measurable function $F:(\bar \chi \sqcup\partial)\times \Omega\ra \bar \chi \sqcup \partial$ such that for all $x\in \bar \chi\sqcup \partial$,
\[
X^x:\Omega\ni \omega\mapsto F(x,\omega)
\]
is a random variable with distribution $X^x\sim K(x,\cdot)$.

When $T<\infty$ we define the random kernel
\[
R^n((t,f);\cdot)=\begin{cases}
\sum_{s=0}^{\tau_{\partial}-1}\Ind(s\leq T)\delta_{(s,(X^{\mu}_{u})_{0\leq u\leq s})}(\cdot),\quad B=1\\
\sum_{s=0}^{\tau_{\partial}-1}\Ind(t+s\leq T)\delta_{(t+s,f\oplus (X^{f(t)}_{u-t})_{t\leq u\leq t+s})}(\cdot),\quad B=0
\end{cases}.
\]
We adopt the convention that $\ast+s:=\ast>\bar T$ for $s\in \Nm$ and $f(\ast):=f$ for $(\ast,f)\in E_{\ast}$. When $T=\infty$ we define the random kernel $R^n$ as
\[
R^n((t,f);\cdot):=\begin{cases}
\sum_{s=0}^{\tau_{\partial}-1}\big[\Ind(s\leq \bar T)\delta_{(s,(X^{\mu}_{u})_{0\leq u\leq s})}(\cdot)+\Ind(s> \bar T)\delta_{(\ast,X_{s})}\big],\quad B=1\\
\sum_{s=0}^{\tau_{\partial}-1}\big[\Ind(t+s\leq \bar T)\delta_{(t+s,f\oplus (X^{f(t)}_{u-t})_{t\leq u\leq t+s})}(\cdot)\\+\Ind(t+s> \bar T)\delta_{(\ast,X^{f(t)}_{s})}\big],\quad B=0
\end{cases}.
\]

These random kernels $R^n$ have common expectation given by the (deterministic) kernel $R:E\ra \calM(E)$, which in the $T<\infty$ case is given by
\[
\begin{split}
R((t,f);\cdot)=\expE[R^n((t,f);\cdot)]=p\sum_{s=0}^{T}\Law_{\mu}((s,(X_{u})_{0\leq u\leq s})\lvert \tau_{\partial}>s)(\cdot)\Pm_{\mu}(\tau_{\partial}>s)\\
+(1-p)\sum_{s=0}^{T-t}\Law_{f(t)}((t+s,f\oplus (X_{u-t})_{t\leq u\leq t+s})\lvert \tau_{\partial}>s)(\cdot)\Pm_{f(t)}(\tau_{\partial}>s).
\end{split}
\]
In the $T=\infty$ case the kernel $R:E\ra \calM(E)$ is given by
\[
\begin{split}
R((t,f);\cdot)=\expE[R^n((t,f);\cdot)]=p\sum_{s=0}^{\infty}\Big[\Ind(s\leq \bar T)\Law_{\mu}((s,(X_{u})_{0\leq u\leq s})\lvert \tau_{\partial}>s)(\cdot)\Pm_{\mu}(\tau_{\partial}>s)\\
+\Ind(s> \bar T)\Law_{\mu}((\ast,X_{s})\lvert \tau_{\partial}>s)(\cdot)\Pm_{\mu}(\tau_{\partial}>s)\Big]\\
+(1-p)\sum_{s=0}^{\infty}\Big[\Ind(t+s\leq \bar T)\Law_{f(t)}((t+s,f\oplus (X_{u-t})_{t\leq u\leq t+s})\lvert \tau_{\partial}>s)(\cdot)\Pm_{f(t)}(\tau_{\partial}>s)\\
+\Ind(t+s> \bar T)\Law_{f(t)}((\ast,X_{s})\lvert \tau_{\partial}>s)(\cdot)\Pm_{f(t)}(\tau_{\partial}>s)\Big].
\end{split}
\]
\item
Non-negative weight kernel 
\begin{equation}
P:E \ni (t,f)\mapsto \frac{1}{\Gamma}\delta_{(t,f)}(\cdot)\in \calM(E).
\end{equation}

\end{itemize}

We therefore define the (deterministic) submarkovian kernel $Q$ as
\begin{equation}
Q:E \ni (t,f)\mapsto \int_{E}P(z;\cdot)R((t,f);dz)=\frac{1}{\Gamma}R((t,f);\cdot)\in \calM(E).
\end{equation}

By Assumption \ref{assum:ext to compact completion}, we see that $m_N$ is an urn process on $E$ satisfying assumptions \cite[$T_{>0}$, (A1), (A2) and (A4)]{Mailler2020}. 

We must therefore verify \cite[Assumption (A3)]{Mailler2020}, that is we must establish the following proposition.

\begin{prop}\label{prop:(A3) verification}
The $ E$-valued continuous-time killed Markov process $(Y^c_t)_{t<\tau^{Y^c}_{\partial}}$ with submarkovian infinitesimal generator $ Q-\text{Id}$ converges uniformly to quasi-equilibrium.
\end{prop}
We defer for later the proof of Proposition \ref{prop:(A3) verification}.

\subsection*{Identifying the limit}

Thus $(Y^c_t)_{t<\tau^{Y^c}_{\partial}}$ admits a unique QSD, $\eta$. Since we have verified Assumption \cite[(A3)]{Mailler2020}, we may invoke \cite[Theorem 1]{Mailler2020}, giving that
\begin{equation}\label{eq:convergence to nu R}
\frac{m_N}{N}(\cdot)\overset{\calM(E)} {\ra}\eta R(\cdot)\quad\text{almost surely}.
\end{equation}

Since QSDs of $(Y_t)_{t<\tau^Y_{\partial}}$ correspond to solutions of
\begin{equation}\label{eq:left e-measure eqn Y}
\alpha(\cdot)=\frac{\alpha Q(\cdot)}{\alpha Q(E)}=\frac{\alpha R(\cdot)}{\alpha R(E)},\quad \alpha\in\calP(E),
\end{equation}
$\eta$ is the unique solution to \eqref{eq:left e-measure eqn Y}.

We assume for the time being that $T<\infty$ and write
\begin{equation}
\tilde{\eta}(\cdot):=\sum_{t=0}^{T}\frac{\gamma_t}{Z}\Law_{\mu}((t,(X_s)_{0\leq s\leq t})\lvert \tau_{\partial}>t)\in\calP(E).
\end{equation} 
We use \eqref{eq:identity for law of path started from other path} and straightforward algebra to calculate
\[
\begin{split}
\tilde{\eta}R(\cdot)=p\sum_{s=0}^T\Pm_{\mu}(\tau_{\partial}>s)\Law_{\mu}((s,(X_u)_{0\leq u\leq s})\lvert \tau_{\partial}>s)
+(1-p)\sum_{s=0}^T\sum_{t=0}^T\Ind(s\geq t)\frac{\gamma_t}{Z}\\ \int_{F([0,t])}\Law_{f(t)}((s,f\oplus (X_{u-t})_{t\leq u\leq s})\lvert \tau_{\partial}>s)(\cdot)\Pm_{f(t)}(\tau_{\partial}>s)d\Law_{\mu}((X_u)_{0\leq u\leq t}\lvert \tau_{\partial}>t)(df)\\
\overset{\eqref{eq:identity for law of path started from other path}}{=}
p\sum_{s=0}^T\Pm_{\mu}(\tau_{\partial}>s)\Law_{\mu}((s,(X_u)_{0\leq u\leq s})\lvert \tau_{\partial}>s)
\\+(1-p)\sum_{s=0}^T\sum_{t=0}^T\Ind(s\geq t)\frac{\gamma_t}{Z}\Pm_{\mu}(\tau_{\partial}>s\lvert \tau_{\partial}>t)\Law_{\mu}((s,(X_u)_{0\leq u\leq s})\lvert \tau_{\partial}>s)(\cdot)\\
=\sum_{s=0}^Tp\Pm_{\mu}(\tau_{\partial}>s)\Big[1+\frac{1-p}{Z}\sum_{t=0}^s\Big(1-\frac{1-p}{Z}\Big)^{-(t+1)}\Big]\Law_{\mu}((s,(X_u)_{0\leq u\leq s})\lvert \tau_{\partial}>s)(\cdot)\\
=\sum_{s=0}^Tp\Pm_{\mu}(\tau_{\partial}>s)\Big[1-\frac{1-p}{Z}\Big]^{-(s+1)}\Law_{\mu}((s,(X_u)_{0\leq u\leq s})\lvert \tau_{\partial}>s)(\cdot)=Z\tilde{\eta}(\cdot)
\end{split}
\]
Therefore $\tilde{\eta}(\cdot)$ is the solution to \eqref{eq:left e-measure eqn Y}, hence $\eta=\tilde{\eta}$. Therefore
\[
\eta R(\cdot)=Z\tilde{\eta}(\cdot)=\sum_{t=0}^T\gamma_t\Law_{\mu}((t,(X_u)_{0\leq u\leq t})\lvert \tau_{\partial}>t).
\]
Combining this with \eqref{eq:convergence to nu R} we have 
\begin{equation}\label{eq:convergence of approx in bar chi}
\frac{1}{N}\sum_{n=1}^N\Ind(t^n_b\leq t<t^n_d)\delta_{f^n_{[0,t]}}(\cdot)\overset{\calM(F([0,t];\bar\chi))}{\ra} \gamma_t\Law_{\mu}((X_s)_{0\leq s\leq t}\lvert \tau_{\partial}>t)(\cdot)\quad \text{almost surely as}\quad N\ra\infty.
\end{equation}
Note that convergence in $\calM(F([0,t];\bar \chi))$ of measures supported on $F([0,t];\chi)$ to a measure supported on $F([0,t];\chi)$ implies convergence in $\calM(F([0,t];\chi))$ (this is easy to prove using the Portmanteau theorem). Since both sides of \eqref{eq:convergence of approx in bar chi} are supported on $F([0,t];\chi)$, we have \eqref{eq:convergence in theorem statement} in the $T<\infty$ case. In the $T=\infty$ case we consider
\[
\tilde{\eta}(\cdot):=\sum_{t=0}^{\bar{T}}\frac{\gamma_t}{Z}\Law_{\mu}((t,(X_u)_{0\leq u\leq t}\lvert \tau_{\partial}>t)(\cdot)+\sum_{t=\bar T+1}^{\infty}\frac{\gamma_t}{Z}\Law_{\mu}(X_t\lvert \tau_{\partial}>t)(\cdot),
\]
and repeat the above calculation to obtain \eqref{eq:convergence in theorem statement} for all $t\leq \bar T$. Since $\bar T\in \Nm$ was arbitrary, we have \eqref{eq:convergence in theorem statement} for all $t\in\Nm$.

We have left only to prove propositions \ref{prop:exist of a fixed point Z inf time horizon} and \ref{prop:(A3) verification}. We begin with Proposition \ref{prop:(A3) verification}.
\subsection*{Proof of Proposition \ref{prop:(A3) verification}}

We seek to check that $(Y^c_t)_{t<\tau^{Y^c}_{\partial}}$ satisfies \cite[(A1) and (A2)]{Champagnat2014}, which implies Proposition \ref{prop:(A3) verification} by \cite[Theorem 2.1]{Champagnat2014}. The former is immediate. We have left to check \cite[(A2)]{Champagnat2014}.

We claim that it is sufficient to show that the $ E$-valued discrete-time killed Markov process $(Y_n)_{n<\tau^{Y}_{\partial}}$ with submarkovian infinitesimal generator $ Q$ satisfies \cite[(A1) and (A2)]{Champagnat2014}. To see this, note that $ Q$ would then have a bounded non-negative right-eigenfunction as given by \cite[Proposition 2.3]{Champagnat2014}, which must then be a bounded non-negative right-eigenfunction for $ Q-\text{Id}$. Furthermore, since $(Y^c_t)_{t<\tau^{Y^c}_{\partial}}$ satisfies \cite[(A1)]{Champagnat2014}, this may then be combined with the existence of the bounded, non-negative right eigenfunction to see that $(Y^c_t)_{t<\tau^{Y^c}_{\partial}}$ satisfies \cite[(A2)]{Champagnat2014}.

Thus it is sufficient to check that $(Y_n)_{n<\tau^{Y}_{\partial}}$ satisfies \cite[(A1) and (A2)]{Champagnat2014}. It is trivial that it satisfies \cite[(A1)]{Champagnat2014} with
\[
\nu(\cdot)=\frac{1}{\expE_{\mu}[\tau_{\partial}\wedge (T+1)]}\sum_{s=0}^{T}\Law_{\mu}((s,(X_{u})_{0\leq u\leq s})\lvert \tau_{\partial}>s)(\cdot)\Pm_{\mu}(\tau_{\partial}>s),
\]
in the language of \cite[(A1)]{Champagnat2014}. It is left to check \cite[(A2)]{Champagnat2014}, for this same $\nu(\cdot)$. We separate the $T<\infty$ and $T=\infty$ cases.

\subsubsection*{The $T<\infty$ case}

We note that we can write $Q=Q_0+Q_1+Q_2$ whereby
\[
\begin{split}
\delta_{(t,f)}Q_0(\cdot)=c\nu(\cdot),\quad c=\frac{p\expE_{\mu}[\tau_{\partial}\wedge (T+1)]}{\Gamma},\\
Q_1=\frac{1-p}{\Gamma}\text{Id},\quad \delta_{(t,f)}Q_2(\cdot)\quad\text{is supported on}\quad \{(t',f')\in E:t'\geq t+1\}.
\end{split}
\]
We write
\[
\begin{split}
\delta_{(t,f)}Q^n1=\delta_{(t,f)}[(Q_1+Q_2)^n+\sum_{m=0}^{n-1}(Q_1+Q_2)^mQ_0Q^{n-m-1}]1\\
=\delta_{(t,f)}(Q_1+Q_2)^n1+c\sum_{m=0}^{n-1}[\delta_{(t,f)}(Q_1+Q_2)^m1][\nu Q^{n-m-1}1].
\end{split}
\]

We observe that $Q_1$ and $Q_2$ commute, and that $Q_2^{T+1}=0$. Thus
\[
(Q_1+Q_2)^m1=\sum_{k=0}^{T+1}\begin{pmatrix}
m\\
k
\end{pmatrix}\Big(\frac{1-p}{\Gamma}\Big)^{m-k}Q_2^k1\leq C(m^{T+1}+1)\Big(\frac{1-p}{\Gamma}\Big)^{m}
\]
for some $C<\infty$. We now observe for $0\leq m\leq n-1$ that 
\[
\nu Q^n 1\geq \nu (Q_0+Q_1)^{m+1}Q^{n-m-1}1\geq \Big(c+\frac{1-p}{\Gamma}\Big)^{m+1}\nu Q^{n-m-1}1.
\]
Therefore combining the above we have
\[
\begin{split}
\delta_{(t,f)}Q^n1\leq C(n^{T+1}+1)\Big(\frac{1-p}{\Gamma}\Big)^{n}+\Big[c\sum_{m=0}^{n-1}C(m^{T+1}+1)\Big(\frac{1-p}{\Gamma}\Big)^{m}\Big(c+\frac{1-p}{\Gamma}\Big)^{-(m+1)}\Big]\nu Q^n 1\\
\leq C'\sum_{m=0}^{n}\Big[(m^{T+1}+1)\Big(\frac{1-p}{\Gamma}\Big)^{m}\Big(c+\frac{1-p}{\Gamma}\Big)^{-(m+1)}\Big]\nu Q^n 1,
\end{split}
\]
for some $C'<\infty$. Therefore there exists a constant $M$, independent of ${(t,f)}$ and $n$, such that $\delta_{(t,f)}Q^n1\leq M\nu Q^n 1$ for all ${(t,f)}\in E$ and $n\geq 1$, which implies that $(Y_t)_{t<\tau^{Y}_{\partial}}$ satisfies \cite[(A2)]{Champagnat2014}.

\subsubsection*{The $T=\infty$ case}

Since the spectral radius of $K$ is less than $1$, \eqref{eq:spectral radius of K less than 1}, the following operator is well-defined,
\[
G=\sum_{n=0}^{\infty}K^n:C_b(\chi)\ra C_b(\chi).
\]
We observe that if $(Y_n)_{n<\tau_{\partial}^Y}=((t^n,f^n))_{n<\tau_{\partial}^Y}$, then $(Z_n)_{n<\tau^Z_{\partial}}:=(f^n(t^n))_{n<\tau^Y_{\partial}}$ is a killed Markov chain on $\bar \chi$ with submarkovian kernel
\[
S(x,\cdot):=\frac{p}{\Gamma}G(\mu,\cdot)+\frac{1-p}{\Gamma}G(x,\cdot).
\]
Since $K$ is compact, $S(x,\cdot)-\frac{1-p}{\Gamma}\delta_x(\cdot)$ defines a positive compact operator on $C_b(\chi)$ with positive spectral radius. Thus by the Krein-Rutman theorem, there exists a non-negative right eigenfunction $\phi\in C_b(\chi)$, which must then be a right eigenfunction for $S$. Since $(Z_n)_{n<\tau^Y_{\partial}}$ also satisfies \cite[(A1)]{Champagnat2014}, being minorised by $\frac{p}{\Gamma}G(\mu,\cdot)$, it must then satisfy \cite[(A2)]{Champagnat2014}: there exists $C<\infty$ such that
\begin{equation}\label{eq:A2 for subchain}
\Pm_x(\tau^Z_{\partial}>n)\leq C\Pm_{\frac{G(\mu,\cdot)}{G(\mu,1)}}(\tau^Z_{\partial}>n)
\end{equation}
for all $n\geq 0$ and $x\in\chi$. Since $(\tilde{t},\tilde{f})\sim \nu$ implies that $\tilde{f}(\tilde{t})\sim \frac{G(\mu,\cdot)}{G(\mu,1)}$, \eqref{eq:A2 for subchain} then implies that
\[
\Pm_{(t,f)}(\tau^Y_{\partial}>n)\leq C\Pm_{\nu}(\tau^Y_{\partial}>n).
\]
\qed

\subsection*{Proof of Proposition \ref{prop:exist of a fixed point Z inf time horizon}}

We consider for $T<\infty$ the function
\begin{equation}\label{eq:fn fT}
f_T:[1,\infty)\ni z\mapsto z-\sum_{s=0}^T\Pm_{\mu}(\tau_{\partial}>s)p\Big[1-\frac{1-p}{z}\Big]^{-(s+1)}\in \Rm.
\end{equation}
Solutions to $f_T(z)=0$, $z\in [1,\infty)$ correspond to solutions to \eqref{eq:fixed point eqn for Z}. We calculate for $T<\infty$,
\[
f_T(1)\leq 1-\Pm_{\mu}(\tau_{\partial}>0)= 0,\quad f_T'(z)\geq 1\quad\text{for all}\quad z\in [1,\infty)\quad \text{and}\quad f_T(z)\ra \infty\quad \text{as}\quad z\ra \infty,
\]

so that there exists a unique solution to \eqref{eq:fixed point eqn for Z}. 

We now consider the $T=\infty$ case. We let $r=r(K)$ be the spectral radius of $K$. We claim that
\begin{equation}\label{eq:ktprob of dying in time t}
\liminf_{t\ra \infty}k^t\Pm_{\mu}(\tau_{\partial}>t)>0\quad\text{for}\quad \frac{1}{r}\leq k<\infty,\quad\lim_{t\ra \infty}k^t\Pm_{\mu}(\tau_{\partial}>t)=0\quad\text{for}\quad k< \frac{1}{r},
\end{equation}
defining $\frac{1}{r}:=+\infty$ in the case that $r(K)=0$.

We have this for $k\neq \frac{1}{r}$ by \eqref{eq:spectral radius of K less than 1}. If $r=0$ then the $k=\frac{1}{r}$ case is vacuous. Otherwise, the Krein-Rutman theorem implies the existence of a positive right eigenfunction $\phi\in C_0(\chi)$, with eigenvalue $r(K)$. This implies that
\[
r^{-t}\Pm_{\mu}(\tau_{\partial}>t)=r^{-t}\mu K^t1\geq \frac{r^{-t}\mu K^t \phi}{\lvert\lvert \phi\rvert\rvert_{\infty}}=\frac{\mu(\phi)}{\lvert\lvert \phi\rvert\rvert_{\infty}}>0 \quad\text{for all}\quad t\geq 0,
\]
implying the $k=\frac{1}{r}$ case.

Therefore the infinite sum $\sum_{s=0}^{\infty}\Pm_{\mu}(\tau_{\partial}>s)p[1-\frac{1-p}{z}]^{-(s+1)}$ is well-defined for
\[
z\in (z_0,\infty)\cap [1,\infty)\quad\text{whereby}\quad z_0:=\frac{1-p}{1-r}<z<\infty,
\]
so that we may define $f_{\infty}$ on $(z_0,\infty)\cap [1,\infty)$ similarly to \eqref{eq:fn fT}. Moreover \eqref{eq:ktprob of dying in time t} gives that $f_{\infty}$ is continuous and strictly increasing. It also gives that $f_{\infty}(z)\ra -\infty$ as $z\downarrow z_0$ if $z_0\geq 1$. If $z_0<1$ (which is the case if $r=0$) we observe that $f_{\infty}(1)\leq 0$ as with $T<\infty$. Therefore we have the existence of a unique solution to \eqref{eq:fixed point eqn for Z} as with $T<\infty$.\qed

This concludes the proof of Theorem \ref{theo:main theorem}. \qed

{\textbf{Acknowledgement:}}  This work was funded by grant 200020 196999 from the Swiss National Foundation. The author would like to thank Michel Benaim for useful discussions on the convergence (or lack thereof) of reinforced processes, leading in particular to Corollary \ref{cor:convergence of reinforced with renewal}.

\bibliography{StochApproxKilledMPLibrary}
\bibliographystyle{plain}
\end{document}